\documentclass[preprint]{elsarticle}

\usepackage[english]{babel}
\usepackage[T1]{fontenc}
\usepackage[utf8]{inputenc}
\usepackage{lmodern}

\usepackage{todonotes}
\usepackage{mathtools}
\usepackage{commath}
\usepackage{natbib}

\usepackage{graphicx}
\usepackage{amsmath,amssymb,amsfonts,amsthm}

\usepackage{hyperref}
\hypersetup{
  pdfauthor={Martin Lukarevski, Hans-Peter Schröcker},pdftitle={Uniqueness of Maximal Inscribed Parabolas and Minimal Circumscribing Horocycles},pdfborder={0 0 0},colorlinks=true,}

\newtheorem{theorem}{Theorem}
\newtheorem{lemma}{Lemma}
\theoremstyle{definition}
\newtheorem{definition}{Definition}
\theoremstyle{remark}
\newtheorem{remark}{Remark}

\DeclareMathOperator{\size}{size}
\DeclareMathOperator{\interior}{int}
\DeclareMathOperator{\cinterior}{\overline{int}}
\DeclareMathOperator{\conv}{conv}

\begin{document}

\begin{frontmatter}
\journal{}

\title{Uniqueness of Maximal Inscribed Parabolas and Minimal Circumscribing Horocycles}
\author{Martin Lukarevski}
\ead{martin.lukarevski@ugd.edu.mk}
\address{Goce Delčev University of Štip, Department of Mathematics and Statistics, Štip, North Macedonia}

\author{Hans-Peter Schröcker}
\ead{hans-peter.schroecker@uibk.ac.at}
\address{University of Innsbruck, Department of Basic Sciences in Engineering Sciences, Innsbruck, Austria}

\begin{abstract}
  We prove existence of three unique ``max-exparabolas'' to a triangle. Each of these parabolas is internally tangent to one edge and the two other sides. Among all like parabolas, it is characterized by having maximal parameter. We use this result to prove a more general uniqueness statement on maximal parabolas in a convex point set. In similar spirit, we demonstrate uniqueness of minimal enclosing horocycles in hyperbolic geometry, provided the enclosed set is sufficiently small.
\end{abstract}


\end{frontmatter}

\section{Introduction}
\label{sec:introduction}

The ellipsoid of minimal volume circumscribing a compact point set and the ellipsoid of maximal volume inscribed into a compact convex set are important objects in convex geometry \cite{gruber11}. Existence of minimal and maximal ellipsoids follows from straightforward compactness arguments. First proofs of uniqueness are due to \cite{behrend38} in dimension two and to \cite{danzer57} for general dimension. Sometimes, minimal and maximal ellipsoids are also named after Charles Loewner (Karel Löwner) who is believed to have found the first proof of uniqueness for the minimal ellipsoid but did not publish it \cite{busemann50,henk12} and Fritz John who characterized the maximal ellipsoid and showed some inequalities related to it, cf. the original \cite{john48} and the re-print~\cite{john13}.

The Loewner-John ellipsoids are not just of theoretical importance but have also found application in such diverse areas as optimization \cite[Chapter~3]{groetschel93}, robotics \cite{dabbene17}, or financial mathematics \cite{shen15}. Hence, also efficient methods for their computation are of interest, for example \cite{albocher09,hayes19,gartsman24,ma24}.

One important property of Loewner-John ellipsoids is their \emph{affine invariance.} It accounts in parts for their theoretical and practical importance and also greatly simplifies proofs of uniqueness. Uniqueness results for minimal circumscribing and maximal inscribed ellipsoids with respect to different functions to measure their size are typically harder to obtain, cf. \cite{gruber08,schroecker08,weber10} for quermassintegrals and more general size functions that satisfy certain convexity properties.

A typical technique for proving uniqueness is indirect. Assuming existence of two different extremal (minimal or maximal) ellipsoids of equal size, existence of a further circumscribing/inscribing ellipsoid that is smaller/larger is shown \cite{danzer57,schroecker08,weber10}. Uniqueness theorems have also been proved for minimal hyperbolas that enclose a set of lines in the plane \cite{schroecker07} and for minimal enclosing ellipses in the elliptic and in the hyperbolic plane \cite{weber12a,weber12b}. These proofs are even more difficult and guarantee uniqueness only under some additional assumptions on the set to be enclosed and the set of conics to minimize over.

In this article we study uniqueness results for classes of conics that have not been considered so far: Parabolas in the Euclidean plane and horocycles in the hyperbolic plane. Both have infinite area and arc length so that typical ways to measure their size fail. However, any two parabolas or horocycles (in the Cayley-Klein model of hyperbolic geometry) can obviously be ``compared in size'' in a Euclidean sense: Unless they are congruent, one can be transformed into the interior of the other by a Euclidean rotation or translation. Our notion of minimal parabolas and maximal horocycles will be based on this concept of size comparison.

Classes of congruent parabolas or horocycles are determined by only one invariant (the \emph{parameter} in case of parabolas), much like the spheres whose size is uniquely determined by the radius. This property is very beneficial in proofs of uniqueness and makes the study of different size functions unnecessary. It should be noted that parabolas and horocycles are unbounded but with only one point at infinity. Hence, unlike ellipsoids, they can be used for conservative (worst case) approximations of certain unbounded point sets.

We continue this text by setting notation and introducing some basic concepts in Section~\ref{sec:preliminaries} before turning to maximal parabolas in Section~\ref{sec:maximal-parabolas}. Our first result (Theorem~\ref{th:exparabolas}) shows that to any triangle three locally maximal parabolas can be associated. These \emph{max-exparabolas} seem to be an interesting new concept of triangle geometry. The uniqueness result of maximal parabolas in more general settings is then dealt with in Theorem~\ref{th:maximal-parabola}. Section~\ref{sec:minimal-horocycles} is dedicated to the study of minimal enclosing horocycles. The main result, Theorem~\ref{th:minimal-horocycle}, claims their uniqueness among all enclosing horocycles whose axis ratio (when viewed as a Euclidean ellipse) does not drop beneath~$\sqrt{2}$.

\section{Preliminaries}
\label{sec:preliminaries}

In this section we collect some known facts about conics in the projective, Euclidean, and hyperbolic plane. We refer the reader to \cite{casas-alvero14}, \cite{glaeser24} for conics in general and \cite[Chapters~25-26]{richter-gebert11} for conics in the hyperbolic plane.

We represent a conic in the real projective plane $\mathbb{P}^2(\mathbb{R})$ by a real symmetric matrix $C \neq 0$ of dimension three by three, keeping in mind that
\begin{itemize}
\item non-zero real multiples of $C$ represent the same conic,
\item the zero matrix does not represent a conic, and
\item real conics represented in this way need not contain real points.
\end{itemize}
The conic $C$ is called \emph{regular} if it can be represented by an invertible matrix.

The \emph{polarity} in a regular conic $C$ is the map that sends a point with projective coordinates $p$ to the line with projective coordinates $C \cdot p$ -- the \emph{polar} of $p$ -- and the line with projective coordinates $u$ to the point $C^{-1} \cdot u$ -- the \emph{pole} of $u$.

A point $p$ is said to be contained in $C$ if $p^t \cdot C \cdot p = 0$. Similarly, a line $u$ is said to be tangent to $C$ if $u^t \cdot C^{-1} \cdot u = 0$. One can identify a regular conic with real points with the set of its points but also with the sets of its tangents. The latter is sometimes referred to as \emph{dual conic.} It is a conic in the dual projective plane where it can be described by the matrix~$C^{-1}$.

\subsection{The Common Interior of Two Conics}
\label{sec:interior}

The \emph{interior $\interior C$} of a regular conic $C$ is the set of all points which are not incident with any tangent of $C$. By $\cinterior C$ we denote its closure in an ambient Euclidean or hyperbolic topology, depending on context. Once $C$ is suitably normalized, we have $\interior C = \{p \in \mathbb{P}^2(\mathbb{R}) \mid p^t \cdot C \cdot p < 0\}$. Assume now that two conics $C_0$, $C_1$ are normalized in that way and, for $t \in [0,1]$, define $C(t) \coloneqq (1-t)C_0 + tC_1$. Then
\begin{equation*}
  p^t \cdot C(t) \cdot p =
  (1-t)p^t \cdot C_0 \cdot p + t p^t \cdot C_1 \cdot p\end{equation*}
so that \emph{interior points of both $C_0$ and $C_1$ are also interior points of~$C(t)$.}

Similarly, we can also consider convex combinations  $D(t) \coloneqq (1-t)D_0 + tD_1$ of dual conics. Provided $D_0$ and $D_1$ are suitably normalized, this yields a dual pencil of conics with the property that every line neither intersecting $D_0$ nor $D_1$ will also not intersect $D(t)$. For the respective primal conics $C_0$, $C_1$, and $C(t)$ this means
\[
    \cinterior C(t) \subset \conv \cinterior C_0 \cup \cinterior C_1.
\]
Here ``$\conv$'' denotes the convex hull of a point set.

Constructions of this type are rather typical in the context of minimal or
maximal conics and quadrics \cite{danzer57,weber10} and we will encounter them as well in our proof of Theorem~\ref{th:maximal-parabola}. Note that, in order to ensure existence of minimal or maximal conics, it is necessary to allow for points to be contained in the \emph{closure} of the primal or dual interior. This will also be our understanding throughout this text.

\subsection{Parabolas and their Size}
\label{sec:parabolas}

Designating one line $\ell_\infty \subset \mathbb{P}^2(\mathbb{R})$ as \emph{line at infinity,} an affine structure is imposed on $\mathbb{P}^2(\mathbb{R})$. Depending on the number of intersection points with $\ell_\infty$, three types of regular conics can be distinguished: Ellipses (zero intersection points), hyperbolas (two intersection points) and parabolas (one intersection point). Here, we focus on the latter.

The condition for the conic $P = (p_{ij})_{i,j=0,1,2}$ to be a parabola is that $\ell_\infty$ is one of its tangents. Using projective coordinates $[x_0,x_1,x_2]$ where $x_0 = 0$ is the equation of $\ell_\infty$, this can be stated as
\begin{equation*}
  (1,0,0) \cdot P^{-1} \cdot (1,0,0)^t = 0
\end{equation*}
or, in terms of the coefficients of $P$, as
\begin{equation*}
  p_{11}p_{22}-p_{12}^2 = 0.
\end{equation*}

While in affine geometry, any two parabolas are equivalent, this is not the case in Euclidean geometry. In fact, any two parabolas correspond in a Euclidean transformation plus a uniform scaling. Thus, the ``size'' of a parabola can be measured as scaling factor relative to a designated ``unit parabola.'' Another way of embodying the same idea uses the parabola's \emph{parameter $p$} (the distance between focal point and directrix \cite[Chapter~2]{glaeser24}) for the purpose of measuring and comparing size. A parabola $P_1$ with parameter $p_1$ can be mapped into the interior of a parabola $P_2$ with parameter $p_2$ by a Euclidean displacement, if and only if $p_1 < p_2$. As to the respective size of $P_1$ and $P_2$ this means that
\[
    \size{P_1} < \size{P_2} \iff p_1 < p_2.
\]

\begin{remark}
  \label{rem:euclidean}
  Using the parameter as measure for the size of a parabola actually mixes concepts of affine and of Euclidean geometry. Consequently, the enclosing parabola of minimal size is not an object of affine geometry.
\end{remark}

In terms of coefficients of the parabola matrix $P$ the parameter $p$ equals
\begin{equation}
  \label{eq:parameter}
  p = \frac{\vert{p_{01}p_{12}-p_{02}p_{11}\vert}}{(p_{11}+p_{22})(p_{11}^2+p_{12}^2)^{1/2}}.
\end{equation}
We will often measure a parabola's size by the squared parameter in order to avoid absolute value in the numerator and the square root in the denominator. In this way we can profit from the algebraicity of the resulting expressions.

\subsection{Horocycles and their Size}
\label{sec:horocycles}

We now consider the Cayley-Klein model of planar hyperbolic geometry. It is embedded in the real projective plane. We use both, homogeneous coordinates $[x_0,x_1,x_2]$ and Cartesian coordinates $(x,y)$ which are related by
\begin{equation*}
  x = \frac{x_1}{x_0},\quad
  y = \frac{x_2}{x_0}.
\end{equation*}
The points of hyperbolic geometry are the interior points of the conic $\mathcal{N}\colon x^2 + y^2 - 1 = 0$. Points of $\mathcal{N}$ are called \emph{ideal points} or \emph{points at infinity.}

Horocycles are defined as regular projective conics with all but one point in the interior of $\mathcal{N}$. In other words, they are Euclidean ellipses that hyperosculate $\mathcal{N}$ in a minor vertex, cf. Figure~\ref{fig:horocycle-proof} for some examples. Both, parabolas in Euclidean geometry and horocycles in hyperbolic geometry share the property of having precisely one point at infinity, so that, arguably, horocycles can be regarded as the ``parabolas of hyperbolic geometry''. This point of view ties together the different results we prove in this paper. It should be noted, however, that other concepts of a parabola in hyperbolic geometry exists, cf. \cite{alkhaldi13,alkhaldi16}.

In order to characterize the equations of a horocycle and to develop a reasonable way to measure its size, we recall that the curvature of the ellipse with equation
\begin{equation*}
  \frac{x^2}{a^2} + \frac{y^2}{b^2} = 1
\end{equation*}
in its minor vertex $(0,b)$ equals $\varkappa = \frac{b}{a^2}$. Thus, if the ellipse is to be a horocycle of the hyperbolic plane with absolute conic $\mathcal{N}$, its semi-major axis $a$ and semi-minor axis $b$ need to be related by $a^2 = b$.

The hyperbolic area of a horocycle is infinite. But, just as in the case of parabolas in Euclidean geometry, we say that horocycle $h_1$ is smaller than horocycle $h_2$ if ``$h_1$ one fits into $h_2$.'' Since any two horocycles correspond in a transformation of hyperbolic geometry (a projective transformation that fixed $\mathcal{N}$ and its interior), the need for this comparison of two horocycles suggests to restrict the transformation group to Euclidean symmetries of $\mathcal{N}$. We already did something similar for parabolas, cf. Remark~\ref{rem:euclidean}. With this notion of ``size comparison'' we can reasonably measure the size of a horocycle $H$ by the function $\size H = a$ with $a$ being the (Euclidean) major axis length $a$.

\begin{remark}
  Section~\ref{sec:parabolas} actually suggests to use the horocycle's parameter (when viewed as ellipse of the Euclidean plane) to measure its size. This is consistent with our approach: The parameter of a conic equals the curvature radius on the conic's principal axis \cite[Theorem~2.1.5]{glaeser24}. For horocycles this value is $a^3$ which is a strictly monotone function of~$a$.
\end{remark}

\section{Maximal Parabolas}
\label{sec:maximal-parabolas}

The purpose of this section is uniqueness statements for maximal parabolas enclosed in a convex point set $\mathbb{F}$ of the Euclidean plane. We first consider the case when $\mathbb{F}$ is bounded by three lines. This case is interesting in its own right but will also be used to prove a more general statement in Section~\ref{sec:uniqueness-maximal-parabolas}.

\subsection{Max-Exparabolas of a Triangle}
\label{sec:exparabolas}

We consider a triangle $\triangle ABC$ with vertices $A$, $B$, $C$ and sides $a$, $b$, $c$ in the Euclidean plane. The triangle's \emph{positive half-planes} are the closed half-planes bounded by $a$, $b$, and $c$, respectively, that contain the third triangle vertex in their interior. The negative half-planes are the complementary planes with the same boundary.

It is well-known that there exist four circles tangent to all three sides of $\triangle ABC$. One of them, the \emph{incircle,} is contained in the intersection of the triangle's positive half planes, while three, the \emph{excircles,} are contained in the intersection of one negative and two positive half-planes. Clearly, incircle and excircles are of maximal size among all circles contained in the respective intersection of three triangle half-planes and tangent to the three triangle sides. This motivates the next definition.

\begin{definition}
  \label{def:exparabolas}
  A parabola $P$ is called an \emph{exparabola} of a triangle $\triangle ABC$ if it is tangent to all three triangle sides. An exparabola is called a \emph{max-exparabola} of $\triangle ABC$ if it is of maximal size (parameter) among all exparabolas contained in the intersection of one negative and two positive half-planes of the triangle.
\end{definition}

\begin{figure}
  \centering
  \includegraphics{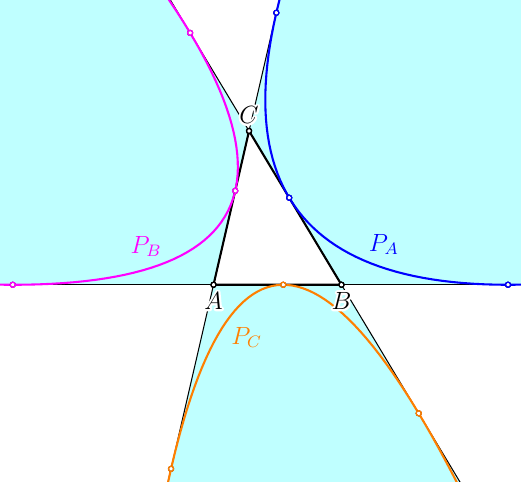}
  \caption{Triangle with three max-exparabolas $P_A$, $P_B$, and $P_C$.}
  \label{fig:exparabolas}
\end{figure}

Figure~\ref{fig:exparabolas} displays a triangle $\triangle ABC$ together with its three max-exparabolas. Existence of precisely three max-exparabolas is the content of the next theorem.

\begin{theorem}
  \label{th:exparabolas}
  There exist precisely three max-exparabolas to a triangle $\triangle ABC$, one in each of the three possible intersections of one negative and two positive half-planes.
\end{theorem}

\begin{proof}
  We assign coordinates
  \[
    A = (a_1,0),\quad
    B = (b_1,0),\quad
    C = (0, c_2)
  \]
  to the triangle's vertices and parametrize the pencil $D(\lambda)$ of dual conics that contain the the line at infinity and the three triangle sides with respective equations
  \begin{equation}
    \label{eq:lines}
    c_2x + b_1y = b_1c_2,\quad
    c_2x + a_1y = a_1c_2,\quad
    y = 0.
  \end{equation}
  This pencil is given by
  \[
    D(\lambda) =
    \begin{bmatrix}
      0 & \lambda - a_1 - b_1 & c_2 \\
      \lambda - a_1 - b_1 & -2a_1b_1 & c_2\lambda \\
      c_2 & c_2\lambda & 0
    \end{bmatrix},
    \quad \lambda \in \mathbb{R}.
  \]
  The dual to the dual conic $D(\lambda)$ is the parabola
  \[
    P(\lambda) =
    \begin{bmatrix}
      -c_2\lambda^2 & c_2\lambda & \lambda^2 - (a_1+b_1)\lambda + 2a_1b_1 \\
      c_2\lambda & -c_2 & \lambda - a_1 - b_1 \\
      \lambda^2 - (a_1+b_1)\lambda + 2a_1b_1 & \lambda - a_1 - b_1 & -c_2^{-1}(\lambda - a_1 - b_1)^2
    \end{bmatrix}.
  \]
  For varying $\lambda$, $P(\lambda)$ describes all parabolas tangent to the lines \eqref{eq:lines} and it can readily be confirmed that the point of tangency with the triangle side $A \vee B$ is given by $(\lambda,0)$.

   By \eqref{eq:parameter}, the squared parameter of $P(\lambda)$ equals
  \begin{equation}
    \label{eq:squared-parameter}
    p^2(\lambda) =
    \frac{4c_2^4(b_1-\lambda)^2(a_1-\lambda)^2}
    {(a_1^2+2a_1b_1-2a_1\lambda+b_1^2-2b_1\lambda+c_2^2+\lambda^2)^3}.
  \end{equation}
  It vanishes for $\lambda \in \{a_1, b_1\}$. These values give rise to two singular parabolas and a third singular parabola is obtained in the limit for $\lambda \to \infty$. The singular parabolas are double lines on the edges of the anticomplementary triangle.

  In order to find local extrema of $p^2(\lambda)$ we compute its derivative:
  \[
    \od{p^2}{\lambda}(\lambda) =
    \frac{8c_2^4(b_1-\lambda)(a_1-\lambda)E(\lambda)}
         {(a_1^2+2a_1b_1-2a_1\lambda+b_1^2-2b_1\lambda+ c_2^2+\lambda^2)^4}
  \]
  where
  \begin{equation}
    \label{eq:cubic}
    E(\lambda) =
    \lambda^3
    -(a_1+b_1)\lambda^2
    +(-a_1^2 + a_1b_1 - b_1^2 - 2c_2^2)\lambda
    +a_1(a_1^2+c_2^2)+b_1(b_1^2+c_2^2).
  \end{equation}
  The up to three roots of $E$ provide us with up to three maximal parabolas inscribed into the triangle $\triangle ABC$. Because of
  \[
    E(a_1)E(b_1) = -(b_1^2+c_2^2)(a_1-b_1)^2(a_1^2+c_2^2) < 0
  \]
  we see that at least one root lies between $a_1$ and $b_1$, that is, at least one of these parabolas is contained in the intersection of the negative half-space through $A$ and $B$ and the positive half-spaces through the remaining triangle sides. For reasons of symmetry, exactly one parabola has this property and the other two parabolas are contained, respectively, in the other negative half spaces.
\end{proof}

\subsection{Uniqueness of Maximal Parabolas}
\label{sec:uniqueness-maximal-parabolas}

Now we look at the uniqueness problem of maximal parabolas in a more general setting.

\begin{theorem}
  \label{th:maximal-parabola}
  If a convex set $\mathbb{F} \subset \mathbb{R}^2$ contains a parabola of maximal size, this parabola is unique.
\end{theorem}

\begin{proof}
  We assume existence of two parabolas $P_0$ and $P_1$ of the same size and both contained in $\mathbb{F}$. Denote by $D_0$ and $D_1$ their duals and set $D(t) \coloneqq (1-t)D_0 + tD_1$. The conic $D(t)$ is a dual parabola and, assuming $D_0$ and $D_1$ are suitably normalized, $D^{-1}(t)$ is contained in $\mathbb{F}$ for any $t \in [0,1]$. Our aim is to show existence of $t_0 \in (0,1)$ such that $\size D(t_0) > \size D_0 = \size D_1$. This is actually not generally but generically true.

  Assume at first that $P_0$ and $P_1$ have four real common tangents. One of them is the line at infinity $\ell_\infty$ so that only three finite tangents remain. No two of them are parallel so that they can be regarded as sides of a triangle. This means that we are actually in the situation of Theorem~\ref{th:exparabolas} where we studied parabolas tangent to the side of a triangle. In particular, the function $p^2(\lambda)$ from Equation~\eqref{eq:squared-parameter}, possibly after a suitable re-parameterization, describes the squared parameter as it varies in $D(t)$. By our discussion, the fuction $p^2(\lambda)$ is pseudo-concave on intervals that separate singular parabolas in $D(t)$ so that existence of a suitable $t_0 \in (0,1)$ follows.

  The remaining cases with less than four real common tangents are not rigid enough for $P_0$ and $P_1$ to be maximal. If the two parabolas have three real tangents, one of them, let us call it $T$, has to be counted with multiplicity two. If $T = \ell_\infty$, the axes of $P_0$ and $P_1$ are parallel. We can translate either of these parabolas into the interior of $\mathbb{G} \coloneqq \conv(\cinterior P_0 \cup \cinterior P_1)$ and enlarge it by a scaling transformation, thus contradicting the assumed maximality of $P_0$ and $P_1$. If $T \neq \ell_\infty$, the argument is similar. At least one of the parabolas can be translated into the interior of $\mathbb{G}$. Since the axes are not parallel, it can be rotated a little  without leaving the interior of $\mathbb{G}$ and then enlarged by scaling. If there are less than three real tangents of $P_0$ and $P_1$, similar arguments apply.
\end{proof}

\begin{remark}
  Our formulation of Theorem~\ref{th:maximal-parabola} carefully avoids issues with possible non-existence of a maximal parabola by just assuming existence. This is not a given, by any means. As one simple example, consider the intersection $\mathbb{F}$ of two half planes. If the half-plane boundaries are parallel, no inscribed parabolas exist. If they are not parallel, $\mathbb{F}$ contains parabolas of arbitrary size.
\end{remark}

\section{Uniqueness of Minimal Horocycles}
\label{sec:minimal-horocycles}

In this section we prove a uniqueness result for enclosing horocycles of minimal size in hyperbolic geometry. Any compact point set $\mathbb{F}$ in the hyperbolic plane can be enclosed by a suitably chosen horocycle and there also exists an enclosing horocycle of minimal size. It turns out that this minimal horocycle is not generally unique but it is under a natural additional assumption on the point set~$\mathbb{F}$.

\begin{lemma}
  \label{lem:horocycles}
  Given two horocycles $H_0$ and $H_1$ of equal size $\size{H_0} = \size{H_1} = a < 2^{-1/2}$ with common interior points, there exists a horocycle $H$ with $\size{H} < \size{H_0} = \size{H_1}$ and $\interior{H_0} \cap \interior{H_1} \subset \interior{H}$.
\end{lemma}

\begin{figure}
  \centering
  \includegraphics[]{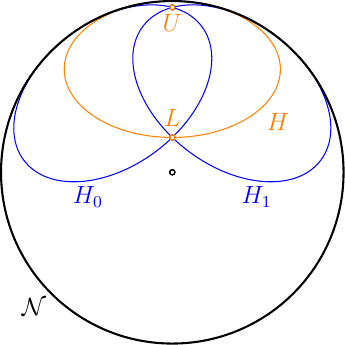}\quad \includegraphics[]{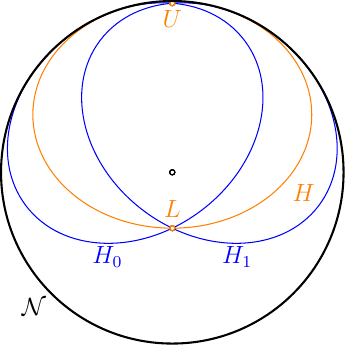}
  \caption{The horocycle $H$ contains the common interior of $H_0$ and $H_1$. In the left image, it is smaller than $H_0$ and $H_1$, in the right image, it is larger.}
  \label{fig:horocycle-proof}
\end{figure}

\begin{proof}
  Since the two horocycles are of equal size, they can be obtained from the horocycle
  \begin{equation*}
    E =
    \begin{bmatrix}
      1-2a^2 & 0 & a^2-1 \\
      0 & a^2 & 0 \\
      a^2-1 & 0 & 1
    \end{bmatrix}
  \end{equation*}
  by respective rotations around $(0,0)$ with angles $\pm\omega$ (cf. Figure~\ref{fig:horocycle-proof}):
  \begin{equation*}
    \begin{aligned}
    H_0 &=
    \begin{bmatrix}
      2(1-2a^2) &
      2(1-a^2)\sin\omega &
      -2(1-a^2)\cos\omega \\
      2(1-a^2)\sin\omega &
      2(a^2\cos^2\omega^2+\sin^2\omega) &
      -(1-a^2)\sin(2\omega) \\
      -2(1-a^2)\cos\omega &
      -(1-a^2)\sin(2\omega) &
      2(a^2\sin^2\omega+\cos^2\omega)
    \end{bmatrix},\\
    H_1 &=
    \begin{bmatrix}
      2(1-2a^2) &
      -2(1-a^2)\sin\omega &
      -2(1-a^2)\cos\omega \\
      -2(1-a^2)\sin\omega &
      2(a^2\cos^2\omega^2+\sin^2\omega) &
      (1-a^2)\sin(2\omega) \\
      -2(1-a^2)\cos\omega &
      (1-a^2)\sin(2\omega) &
      2(a^2\sin^2\omega+\cos^2\omega)
    \end{bmatrix}.
    \end{aligned}
  \end{equation*}
  Provided the horocycles $H_0$ and $H_1$ have common interior points (that is, $\vert\omega\vert$ is not too large), they intersect in the two real points $L = (0, \ell)$, $U = (0, u)$ where
  \begin{align}
    \ell &= \frac{\cos\omega-a^2\cos\omega-a\sqrt{2a^2-a^2\cos^2\omega-\sin^2\omega}}
  {a^2\sin^2\omega+\cos^2\omega}, \label{eq:l}\\
    u    &= \frac{\cos\omega-a^2\cos\omega+a\sqrt{2a^2-a^2\cos^2\omega-\sin^2\omega}}
  {a^2\sin^2\omega+\cos^2\omega}. \label{eq:u}
  \end{align}
  Obviously, $\ell < u$. We define $H$ to be the horocycle that is tangent to $\mathcal{N}$ in $(0,1)$ and contains $L$ and we will show two things:
  \begin{enumerate}
    \item $\size(H) < \size(H_0) = \size(H_1)$
    \item $X \in \interior{H_0} \cap \interior{H_1}$ $\implies$ $X \in \interior{H}$
  \end{enumerate}
  (cf. Figure~\ref{fig:horocycle-proof} which also illustrates that $\size(H) > \size(H_0) = \size(H_1)$ if $a > 2^{-1/2}$).

  \emph{Proof of 1.}
  The statement follows from comparing the (Euclidean) minor axis lengths of the (Euclidean) ellipses $H$, $H_0$, and $H_1$. The minor axis length of $H$ equals $1-\ell$ while that of $H_0$ and of $H_1$ equals $2a^2$. Thus, we need to show $2a^2 + \ell - 1 > 0$. By \eqref{eq:l} and after substituting $\omega = 2\arctan t$ this is equivalent to
  \begin{equation}
    \label{eq:size}
    a(t^2+1)\sqrt{q} < 2 - 3a^2 - a^2t^4 - 2(2a^2-1)^2t^2
  \end{equation}
  where $q = (t^4+6t^2+1)a^2-4t^2 > 0$. We claim that the right-hand side $R$ of \eqref{eq:size} is positive under the assumptions $0 < t < 1$ and $a^2 < \frac{1}{2}$. In order to see this, note that $R$ it is strictly monotone decreasing in $t$ because the coefficients of $t^4$ and of $t^2$ are non-positive and at least one of them is strictly negative. Therefore, $R > 0$ if $R$ is positive after substituting $t = 1$. Indeed, we have $R\mid_{t = 1} = 4a^2(1-2a^2) > 0$.

  Now, we may square both sides of \eqref{eq:size} to get rid of the square root and collect everything again on the left. Denoting the left-hand size of \eqref{eq:size} by $L$, this results in
  \[
    L^2 - R^2 = 4(1-a^2)(2a^2t^2+1)(4a^2t^2+(t^2-1)^2)(2a^2-1) < 0,
  \]
  which is true because $a^2 < \frac{1}{2}$ and $0 < t < 1$.

  \emph{Proof of 2.}
  Next, we consider a point $(x, y)$. It is contained in $\interior{H_0} \cap \interior{H_1}$ if and only if
  \begin{multline}
    \label{eq:int0}
    1 - 2a^2
    +(a^2x^2+y^2)\cos^2\omega
    +(a^2y^2+x^2)\sin^2\omega\\
    -2y(1-a^2)(x\sin\omega+1)\cos\omega
    +2x(1-a^2)\sin\omega < 0
  \end{multline}
  and
  \begin{multline}
    \label{eq:int1}
    1-2a^2
    +(a^2x^2+y^2)\cos^2\omega
    +(a^2y^2+x^2)\sin^2\omega\\
    +2y(1-a^2)(x\sin\omega-1)\cos\omega
    -2x(1-a^2)\sin\omega < 0.
  \end{multline}
  We need to show that \eqref{eq:int0} and \eqref{eq:int1} imply that $(x,y)$ is also contained in $\interior H$, that is,
  \begin{multline}
    \label{eq:int2}
    (x^2+2y-2)(a\sqrt{2a^2-a^2\cos^2w-\sin^2w}+(a^2-1)\cos w) \\
    +(x^2+2y^2-2y)(1+(a^2-1)\sin^2w) < 0.
  \end{multline}

  In order to make the inequalities \eqref{eq:int0} and \eqref{eq:int1} algebraic, we substitute $\omega = 2\arctan(t)$ whence \eqref{eq:int0} and \eqref{eq:int1} become
  \begin{multline}
    \label{eq:int0b}
    (4a^2t^2+t^4-2t^2+1)y^2  - 2(t^2-1)(a^2-1)(t^2+2tx+1)y \\
    +((t^2-1)^2x^2-(4t^3+4t)x-2(t^2+1)^2)a^2 + (t^2+2tx+1)^2 < 0
  \end{multline}
  and
  \begin{multline}
    \label{eq:int1b}
    (4a^2t^2+t^4-2t^2+1)y^2-2(t^2-1)(a^2-1)(t^2-2tx+1)y\\
    +((t^2-1)^2x^2+(4t^3+4t)x-2(t^2+1)^2)a^2+(t^2-2tx+1)^2 < 0
  \end{multline}
  respectively. The same substitution turns inequality ~\eqref{eq:int2} into
  \begin{multline}
    \label{eq:int2b}
    a(t^2+1)(x^2+2y-2)\sqrt{q} + ((2-x^2-2y)a^2+2x^2+2y^2-2)t^4 \\
    +4(a^2-\tfrac{1}{2})(x^2+2y^2-2y)t^2 + (x^2+2y-2)a^2+2(y-1)^2 < 0.
  \end{multline}
  Since the coefficient of $\sqrt{q}$ is negative for $x^2+y^2-1 < 0$, we may subtract $a(t^2+1)(x^2+2y-2)\sqrt{q}$ on both sides of \eqref{eq:int2b}, square, and collect everything on the left-hand side again to arrive at $-4(4a^2t^2+(t^2-1)^2)k < 0$ where
  \begin{multline}
    \label{eq:int2c}
    k =
    (x^2+y^2-1)((x^2+2y-2)a^2-x^2-y^2+1)t^4 \\
    -4(x^2+y^2-1)(y-1)^2(a^2-\tfrac{1}{2})t^2
    -(y-1)^2(y^2+(2a^2-2)y+1+(x^2-2)a^2).
  \end{multline}
  Using quantifier elimination (we employ the \texttt{Reduce} command of Mathematica 14.2 by Wolfram Research Inc.), we can show that this is indeed true provided \eqref{eq:int0b} and \eqref{eq:int1b} hold in addition to
  \[
      0 < a < 2^{-1/2},\quad
      0 < t < 1,\quad\text{and}\quad
      x^2 + y^2 < 1.\qedhere
  \]
\end{proof}

An immediate consequence of Lemma~\ref{lem:horocycles} is

\begin{theorem}
  \label{th:minimal-horocycle}
  Given a compact point set $\mathbb{F}$ in the hyperbolic plane that can be enclosed by a horocycle of size $a < 2^{-1/2}$, there exists a unique horocycle of minimal size that encloses~$\mathbb{F}$.
\end{theorem}
\begin{proof}
  The assumptions of the theorem guarantee existence of an enclosing horocycle of $\mathbb{F}$ with $\size(\mathbb{F}) < 2^{-1/2}$. Assume that $H_0$, $H_1$ are both enclosing horocycles of minimal size. Then $\size(H_0) = \size(H_1) < 2^{-1/2}$. But this is a contradiction to Lemma~\ref{lem:horocycles}.
\end{proof}

\begin{remark}
  It is not possible to improve the bound on $a$ in Theorem~\ref{th:minimal-horocycle}. Consider Figure~\ref{fig:horocycle-proof-2} where the situation of the proof of Lemma~\ref{lem:horocycles} is depicted for the case $a = 2^{-1/2}$. All three horocycles, $H_0$, $H_1$, and $H$ are of equal size as all of them contain the (Euclidean) center $L$ of $\mathcal{N}$. But so does any minimal enclosing horocycle of $\cinterior{H_0} \cap \cinterior{H_1}$. Therefore, $H_0$, $H_1$, $H$ and infinitely many more horocycles are minimal. Some of them are displayed in Figure~\ref{fig:horocycle-proof-2} in gray color.
\end{remark}

\begin{figure}
  \centering
  \includegraphics[]{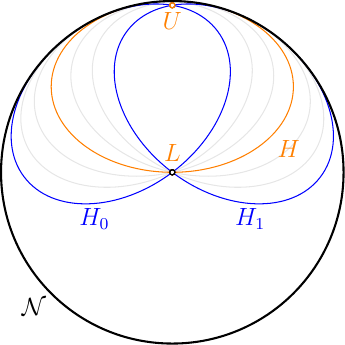}
  \caption{Point set with non-unique minimal enclosing horocycle.}
  \label{fig:horocycle-proof-2}
\end{figure}

\section{Conclusions}
\label{sec:conculsions}

We have proved uniqueness results for maximal enclosed parabolas in the Euclidean plane and of minimal circumscribing horocycles in the hyperbolic plane. These results pose the natural question for uniqueness results for minimal circumscribing parabolas and maximal enclosed horocycles. Unfortunately, neither our proof of Theorem~\ref{th:maximal-parabola} nor of Theorem~\ref{th:minimal-horocycle} can be easily adapted to the other case. The adaptation of the proof of Theorem~\ref{th:maximal-parabola} fails to the non-existence of linear dual pencils or horocycles, the adaptation of Theorem~\ref{th:minimal-horocycle} fails due to the non-uniqueness of a parabola through just one extremal point~$L$.

The result (and also the proof) of Theorem~\ref{th:exparabolas} on existence of maximal parabolas inscribed into a triangle is needed in order to prove Theorem~\ref{th:maximal-parabola}, the general uniqueness result on maximal parabolas. It is, however, of interest in its own right as it introduces the concept of max-exparabolas to triangle geometry. Their position within modern triangle geometry is yet to be clarified. This is topic of future research.

\section*{Acknowledgment}

This research is funded by the Ministry of Education and Science of Republic of North Macedonia: Bilateral Austrian-Macedonian scientific Project No.~20-8436/20 with the title ``Estimates for Ellipsoids in Classical and Non-Euclidean Geometries'' and by OeAD-GmbH Austria's Agency for Education and Internationalization, Project No.~MK05/2024, Project Title ``Estimates for Ellipsoids in Classical and Non-Euclidean Geometries''.

\bibliographystyle{elsarticle-num}

\begin{thebibliography}{10}
\expandafter\ifx\csname url\endcsname\relax
  \def\url#1{\texttt{#1}}\fi
\expandafter\ifx\csname urlprefix\endcsname\relax\def\urlprefix{URL }\fi
\expandafter\ifx\csname href\endcsname\relax
  \def\href#1#2{#2} \def\path#1{#1}\fi

\bibitem{gruber11}
P.~M. Gruber, {John} and {Loewner} ellipsoids, Discrete Comput. Geom. 46~(4)
  (2011) 776--788.
\newblock \href {https://doi.org/10.1007/s00454-011-9354-8}
  {\path{doi:10.1007/s00454-011-9354-8}}.

\bibitem{behrend38}
F.~Behrend, \href{http://eudml.org/doc/159971}{Über die kleinste
  umbeschriebene und die größte einbeschriebene {Ellipse} eines konvexen
  {Bereichs}}, Math. Ann. 115 (1938) 379--411.
\newline\urlprefix\url{http://eudml.org/doc/159971}

\bibitem{danzer57}
L.~Danzer, D.~Laugwitz, H.~Lenz, Über das {Löwnersche Ellipsoid} und sein
  {Analogon} unter den einem {Eikörper} einbeschriebenen {Ellipsoiden}, Arch.
  Math. (Basel) 8~(3) (1957) 214--219.
\newblock \href {https://doi.org/10.1007/bf01899996}
  {\path{doi:10.1007/bf01899996}}.

\bibitem{busemann50}
H.~Busemann, The foundations of {Minkowskian} geometry, Comment. Math. Helv.
  24~(1) (1950) 156--187.
\newblock \href {https://doi.org/10.1007/bf02567031}
  {\path{doi:10.1007/bf02567031}}.

\bibitem{henk12}
M.~Henk, Löwner-John ellipsoids, EMS Press, 2012, pp. 95--106.
\newblock \href {https://doi.org/10.4171/dms/6/15}
  {\path{doi:10.4171/dms/6/15}}.

\bibitem{john48}
F.~John, Studies and Essays Presented to {R. Courant} on his 60th Birthday,
  Interscience Publishers, Inc., New York, 1948, Ch. Extremum Problems with
  Inequalities as Subsidiary Conditions., pp. 187--204.

\bibitem{john13}
F.~John, Extremum Problems with Inequalities as Subsidiary Conditions, Springer
  Basel, 2013, pp. 197--215.
\newblock \href {https://doi.org/10.1007/978-3-0348-0439-4_9}
  {\path{doi:10.1007/978-3-0348-0439-4_9}}.

\bibitem{groetschel93}
M.~Grötschel, L.~Lovász, A.~Schrijver, Geometric Algorithms and Combinatorial
  Optimization, Algorithms and Combinatorics, Springer Berlin Heidelberg, 1993.
\newblock \href {https://doi.org/10.1007/978-3-642-78240-4}
  {\path{doi:10.1007/978-3-642-78240-4}}.

\bibitem{dabbene17}
F.~Dabbene, D.~Henrion, C.~M. Lagoa, Simple approximations of semialgebraic
  sets and their applications to control, Automatica 78 (2017) 110--118.
\newblock \href {https://doi.org/10.1016/j.automatica.2016.11.021}
  {\path{doi:10.1016/j.automatica.2016.11.021}}.

\bibitem{shen15}
W.~Shen, J.~Wang, Transaction costs-aware portfolio optimization via fast
  {Lowner-John} ellipsoid approximation, Proc. AAAI Conf. Artif. Intell. 29~(1)
  (2015).
\newblock \href {https://doi.org/10.1609/aaai.v29i1.9453}
  {\path{doi:10.1609/aaai.v29i1.9453}}.

\bibitem{albocher09}
D.~Albocher, G.~Elber, On the computation of the minimal ellipse enclosing a
  set of planar curves, in: 2009 IEEE International Conference on Shape
  Modeling and Applications, IEEE, 2009, pp. 185--192.
\newblock \href {https://doi.org/10.1109/smi.2009.5170147}
  {\path{doi:10.1109/smi.2009.5170147}}.

\bibitem{hayes19}
M.~J.~D. Hayes, Z.~A. Copeland, P.~J. Zsombor-Murray, A.~Gfrerrer, Largest Area
  Ellipse Inscribing an Arbitrary Convex Quadrangle, Springer International
  Publishing, 2019, pp. 239--248.
\newblock \href {https://doi.org/10.1007/978-3-030-20131-9_24}
  {\path{doi:10.1007/978-3-030-20131-9_24}}.

\bibitem{gartsman24}
R.~Gartsman, N.~Linial, On the {Löwner-John} ellipsoids of the metric
  polytope, Discrete Comput. Geom. (2024).
\newblock \href {https://doi.org/10.1007/s00454-024-00703-4}
  {\path{doi:10.1007/s00454-024-00703-4}}.

\bibitem{ma24}
L.~Ma, Y.~Zhou, Construction of the ellipse with maximum area inscribed in an
  arbitrary convex quadrilateral, Comput. Aided Geom. Design 111 (2024) 102323.
\newblock \href {https://doi.org/10.1016/j.cagd.2024.102323}
  {\path{doi:10.1016/j.cagd.2024.102323}}.

\bibitem{gruber08}
P.~M. Gruber, Application of an idea of {Voronoĭ} to {John} type problems,
  Adv. Math. 218~(2) (2008) 309--351.
\newblock \href {https://doi.org/10.1016/j.aim.2007.12.005}
  {\path{doi:10.1016/j.aim.2007.12.005}}.

\bibitem{schroecker08}
H.-P. Schröcker, Uniqueness results for minimal enclosing ellipsoids, Comput.
  Aided Geom. Design 25~(9) (2008) 756--762.
\newblock \href {https://doi.org/10.1016/j.cagd.2008.07.007}
  {\path{doi:10.1016/j.cagd.2008.07.007}}.

\bibitem{weber10}
M.~J. Weber, H.-P. Schröcker, Davis' convexity theorem and extremal
  ellipsoids, Beitr. Algebra Geom. 51~(1) (2010) 263--274.

\bibitem{schroecker07}
H.-P. Schröcker, Minimal enclosing hyperbolas of line sets, Beitr. Algebra
  Geom. 48~(2) (2007) 367--381.

\bibitem{weber12a}
M.~J. Weber, H.-P. Schröcker, Minimal area conics in the elliptic plane, Adv.
  Geom. 14~(4) (2012) 665--684.
\newblock \href {https://doi.org/10.1515/advgeom-2012-0010}
  {\path{doi:10.1515/advgeom-2012-0010}}.

\bibitem{weber12b}
M.~J. Weber, H.-P. Schröcker, Minimal area ellipses in the hyperbolic plane,
  Beitr. Algebra Geom. 54~(1) (2012) 181--200.
\newblock \href {https://doi.org/10.1007/s13366-012-0112-8}
  {\path{doi:10.1007/s13366-012-0112-8}}.

\bibitem{casas-alvero14}
E.~Casas-Alvero, Analytic Projective Geometry, European Mathematical Society,
  2014.

\bibitem{glaeser24}
G.~Glaeser, H.~Stachel, B.~Odehnal, The Universe of Conics. From the ancient
  Greeks to 21st century developments, 2nd Edition, Springer Spektrum, Berlin,
  Heidelberg, 2024.
\newblock \href {https://doi.org/10.1007/978-3-662-70306-9}
  {\path{doi:10.1007/978-3-662-70306-9}}.

\bibitem{richter-gebert11}
J.~Richter-Gebert, Perspectives on Projective Geometry, Springer, Berlin,
  Heidelberg, 2011.
\newblock \href {https://doi.org/10.1007/978-3-642-17286-1}
  {\path{doi:10.1007/978-3-642-17286-1}}.

\bibitem{alkhaldi13}
A.~Alkhaldi, N.~J. Wildberger, The parabola in universal hyperbolic geometry
  {I}, KoG 17 (2013) 14--42.

\bibitem{alkhaldi16}
A.~Alkhaldi, N.~J. Wildberger, The parabola in universal hyperbolic geometry
  {II}: {Canonical} points and the $\mathcal{Y}$-conic, J. Geom. Graph. 20~(1)
  (2016) 1--11.

\end{thebibliography}

\end{document}